\documentclass[a4paper,12pt]{article}

\usepackage[tbtags]{amsmath}
\usepackage{amsfonts,amssymb,amsthm}
\usepackage{eucal}
\usepackage{natbib}
\usepackage{url}
\usepackage[all]{xy}

\newtheorem{theorem}{Theorem}
\newtheorem{proposition}[theorem]{Proposition}

\newtheorem{definition}[theorem]{Definition}

\newcommand{\curlyX}{\mathcal{X}}
\newcommand{\curlyY}{\mathcal{Y}}
\newcommand{\ldef}{\mathrel{:=}}

\newcommand{\boost}{\vphantom{\big)}}

\newcommand{\qm}{\buildrel \rm q.m. \over \longrightarrow}

\newcommand{\rd}{\textrm{d}}
\newcommand{\ip}[1]{\left\langle {#1} \right\rangle}
\newcommand{\norm}[1]{\lVert {#1} \rVert}
\newcommand{\Order}{\textit{O}}

\newcommand{\E}{\mathop{\textsf{E}}\nolimits}
\newcommand{\Var}{\mathop{\textsf{var}}\nolimits}
\newcommand{\Sd}{\mathop{\textsf{sd}}\nolimits}
\newcommand{\Cov}{\mathop{\textsf{cov}}\nolimits}
\newcommand{\Corr}{\mathop{\textsf{corr}}\nolimits}
\newcommand{\given}{\mathbin{\vert}\nolinebreak}

\title{A representation theorem for stochastic processes with
  separable covariance functions, and its implications for emulation}

\author{Jonathan Rougier\thanks{\sloppy{}Department of Mathematics, University
    Walk, Bristol BS8 1TW, UK.
    Email \texttt{j.c.rougier@bristol.ac.uk}.}}

\date{Originally written Oct 2012}

\begin{document}

\maketitle

\begin{abstract}\noindent
Many applications require stochastic processes specified on two- or higher-dimensional domains; spatial or spatial-temporal modelling, for example.  In these applications it is attractive, for conceptual simplicity and computational tractability, to propose a covariance function that is \emph{separable}; e.g.\ the product of a covariance function in space and one in time.  This paper presents a representation theorem for such a proposal, and shows that all processes with continuous separable covariance functions are second-order identical to the product of second-order uncorrelated processes. It discusses the implications of separable or nearly separable prior covariances for the statistical emulation of complicated functions such as computer codes, and critically reexamines the conventional wisdom concerning emulator structure, and size of design.


\end{abstract}

\section{Introduction}
\label{sec:intro}

Many statistical applications require covariance functions expressed
over two or more dimensions.  Spatial and spatial-temporal modelling
are obvious applications, where the number of dimensions is typically
two or three.  Emulating deterministic functions with Gaussian
processes, part of the general field of computer experiments
\citep[see, e.g.,][]{santner03}, will often require five or ten
dimensions---sometimes more.

Covariance functions, being non-negative definite symmetric, are
highly structured, and one does not hit upon them by chance.  In some
cases, domain symmetry can be used to simplify the problem.  Thus, if
two of the dimensions are spatial and there is no preferential
direction, then an isotropic covariance function can be defined on the
basis of distance alone, reducing a two-dimensional problem to a
one-dimensional one.  This device is not available in computer
experiments, where each dimension represents an input to the code, and
there is no reason why two different inputs should even have the same
units, let alone have a symmetric effect on the code output.

Therefore it is often attractive to take advantage of the general
result that a $p$-dimensional covariance function can be built up as
the product of $p$ one-dimensional covariance functions.  This product
form is termed a `separable covariance function'.  At the very least,
all such covariance functions satisfy the necessary conditions of
being non-negative definite symmetric.  There are other advantages of
this approach, discussed in section~\ref{sec:separable}.
Section~\ref{sec:separable} also presents the restrictions on the
conditional and marginal correlation functions which follow directly
from the separability of the covariance function.

Of more general interest is whether a separable covariance function
provides any restrictions on the underlying stochastic process itself.
Section~\ref{sec:result} provides a complete answer to this question,
giving a representation theorem on the underlying process which holds
if and only if the covariance function is separable.  There is also a
close relationship between separable covariance functions and a
product form for the underlying process, e.g.\ the situation in which
$F(x, y)$ might be written as $F_x(x) \times F_y(y)$.  It is
well-known that if $F_x$ and $F_y$ are probabilistically independent,
then $F$ has a separable covariance function.
Section~\ref{sec:products} provides a converse result, in terms of
second-order properties.  This allows us to `explain' the restrictions
of the conditional and marginal correlations in terms of the product
form for $F$.

The main implications of these results are for the emulation of
complex computer codes, discussed in section~\ref{sec:emulators}.
Here it is completely standard to use separable covariance functions
as a large component of the emulator, and, indeed, the conventional
wisdom is that the entire emulator may be constructed in this fashion.
This advice is critically analysed using the representation theorem,
allowing us to identifying why it might perform well in many
applications, and when it breaks down.  Finally,
section~\ref{sec:summary} concludes with a brief summary.

\section{Separable covariance functions}
\label{sec:separable}

Consider a real-valued stochastic process $F$ with domain $\curlyX
\times \curlyY$.  The covariance function of $F$ is denoted
\begin{equation}
  \label{eq:kappa}
  \kappa \{ (x, y), (x', y') \} \ldef \Cov \{ F(x, y), F(x', y') \} .
\end{equation}
If $F$ has a separable covariance function then $\kappa$ factorises into
the product of a function in $(x, x')$ and a function in $(y, y')$,
denoted
\begin{equation}
  \label{eq:kappasep}
  \kappa\{ (x, y), (x', y') \} = \kappa_x(x, x') \, \kappa_y(y, y') .
\end{equation}
For clarity, in this paper this is stated as ``the covariance function
is separable'', noting that separability as used here should not be
confused with the property of separability of metric spaces
\citep[see, e.g.,][chapter~1]{kreyszig78} and the related property of
separability of stochastic processes \citep[see,
e.g.,][sec.~35]{loeve60}.  Nor should it be confused with the notion
of separability used in \citet{genton04}, which considers the case
where $\kappa_x(x, x')$ can be written as $r_1(x) r_2(x')$.

There are two principal advantages when the covariance function is
separable.  First, it can be hard to specify a non-negative definite
function jointly over a two- or higher-dimensional domain, and it is
very useful that such functions can be built up as products of simpler
functions.  This is particularly true in the case where the covariance
function contains parameters that need to be estimated, because in
this case the parameters separate cleanly into $x$-parameters and
$y$-parameters.  This is the motivation for using separable covariance
functions for emulating complex computer codes, as discussed in more
detail in section~\ref{sec:emulators}.  Note that separability of the
covariance function is not preserved under rotation; it insists on a
preferential set of directions in the input space, aligned with the
axes.  The exception is the squared exponential correlation function
with a common correlation length, i.e.\ (for the stationary case)
\begin{equation}
  \label{eq:3}
  \begin{split}
    \kappa \{ (x, y), (x', y') \}
    & = \sigma^2 \exp \{ -\theta^2 (x - x')^2\} \,\exp\{ -\theta^2 (y - y')^2\} \\
    & = \sigma^2 \exp \{ -\theta^2 h[ (x, y), (x', y')]^2 \}
  \end{split}
\end{equation}
where $h[ \cdot, \cdot]$ denotes Euclidean distance.  This is a very
popular choice in computer experiments, originating in the papers of
\citet{sacks89} and \citet{currin91}, although different correlation
lengths are used in each direction.

Second, in situations where the process $F$ is observed on a grid, the
variance matrix of the observations has Kronecker product form, and
hence is much more easily inverted.  If the grid has $m \times n$
points, then this converts an $\Order\big( (m + n)^3 \big)$
calculation into a $\Order(m^3) + \Order(n^3)$ calculation.  This
result is widely used in space-time kriging.  \citet{genton07}, for
example, presents a method for finding separable approximations to
space-time variance matrices, while \citet{li07} present a
non-parametric test for separability \citep[see also][]{li08}.
\citet{gneiting07} review general approaches to modelling
spatial-temporal processes, including an example of fitting a
separable covariance function and a comparison with other structured
approaches.

However, there is a price for these benefits: separability of the
covariance function is a strong constraint on the nature of $F$.  In
the supporting material for \citet{koh01}, \citet{ohagan98a} considers
the implication of a separable covariance function for the conditional
covariance of a Gaussian process, namely that
\begin{equation}
  \label{eq:ohagan}
  \Cov \{ F(x, y), F(x', y') \given F(x', y) \} = 0 
\end{equation}
(see Figure~\ref{fig:Cov}).  \citeauthor{ohagan98a} is able to provide
a representation theorem for the covariance function of processes
having this type of conditional covariance structure.  This is related
to the separability of the covariance function of a transformed
process.  A similar result to \eqref{eq:ohagan} holds in the more
general Bayes linear case, where conditioning is replaced by
projection \citep{goldstein07}.

\begin{figure}
  \centering
  \setlength{\unitlength}{10mm}
  \begin{picture}(6,6)(1,2)
    \put(3,3){\circle*{0.15}}\put(2,3){\makebox(0,0){$(x,y)$}}
    \put(5,3){\circle*{0.15}}\put(6,3){\makebox(0,0){$(x',y)$}}
    \put(5,5){\circle*{0.15}}\put(6,5){\makebox(0,0){$(x',y')$}}
  \end{picture}
  \caption{$\Cov \{ F(x, y), F(x', y') \mathbin{\vert} F(x', y) \} = 0$ when
    the covariance function is separable.}
  \label{fig:Cov}
\end{figure}
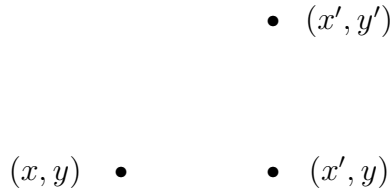

\citet{cressie99} consider the implications of the separability of the
covariance function when $F$ is a spatial-temporal process.  In general,
separability of the covariance function implies that
\begin{equation}
  \label{eq:cressie}
  \Corr \{ F(x, y), F(x, y') \}
  = \frac{ \kappa_y(y, y') }{ \sqrt{ \kappa_y(y, y) \, \kappa_y(y', y') } }
\end{equation}
for all values of $x$.  In other words, if $x$ represents location and
$y$ represents time, then the temporal correlation structure cannot
vary spatially.  \citeauthor{cressie99} conclude that this separable
covariance function ``does not model space-time interaction''
(p.~1331).  A general concern about the absence of interaction has
lead to substantial effort being devoted to developing flexible and
parametric stochastic processes with non-separable covariance
functions \citep[see,
e.g.,][]{cressie99,iaco02,gneiting02,stein05,kent11}.

Note, to avoid confusion, that this type of interaction is different
from that modelled in decompositions of the type
\begin{equation}
  \label{eq:anova}
  f(x, y) = \alpha_0 + \alpha_1(x) + \alpha_2(y) + \alpha_{12}(x, y)
\end{equation}
where $f$ is a deterministic function \citep[see, e.g.,][]{owen97}.
Here $\alpha_{12}(x, y)$ would be the interaction term.  But $F$ is a
stochastic process, \emph{not} a deterministic function.  If $f$ is a
realisation of $F$ then it will almost certainly have an $\alpha_{12}$
term.  When talking of interactions in the stochastic process $F$, we
need to refer to the properties of the distribution of $F$.  Hence, if
the focus is on second-order properties, we must consider interactions
in terms of the properties of the covariance and correlation
functions.  So, `no interactions in $F$' means that the correlation
function of $F$ is invariant to the value of $x$ when considered along
$y$ (and \textit{vice versa}).


\section{Representation theorem}
\label{sec:result}

These preliminaries are from \citet{loeve60}, chapter~10. Consider the
set of all real random quantities with finite second moments, denoted
$F, F', \dots$.  Identify each random quantity with its equivalence
class, where two random quantities are equivalent if they are
identical, or differ only on a set of measure zero.  These equivalence
classes represent points in a Hilbert space, with inner product
$\ip{F, F'} = \E(F F')$.  For simplicity, and without loss of
generality, consider all random quantities to be centred, so that the
inner product represents the covariance, and orthogonal random
quantities are uncorrelated.  In this case the Hilbert space has norm
and distance
\begin{displaymath}
  \norm{ F } = \Sd(F) \quad \text{and} \quad d(F, F') = \Sd( F - F' ) ,
\end{displaymath}
where `$\Sd$' denotes `standard deviation'.  (These are just for
orientation, they are not used in what follows.)  Convergence in this
Hilbert space is equivalent to convergence in quadratic mean, written
here as
\begin{equation}
  F^{(n)} \qm F \iff \E\{ (F^{(n)} - F)^2 \} \to 0 .
\end{equation}

Now within this Hilbert space consider a family of random quantities
indexed by the tuple $(x, y) \in \curlyX \times \curlyY$, where
$\curlyX$ and $\curlyY$ are both closed and bounded intervals of the
real line.  This family is termed a \emph{stochastic process}.  The
covariance function of this stochastic process is
\begin{equation}
  \label{eq:kappa1}
  \kappa\{ (x, y), (x', y') \} = \ip{ F(x, y), F(x', y') } = \E\{ F(x, y) F(x', y') \}
\end{equation}
where, necessarily, $\kappa$ is symmetric and non-negative definite.
This paper investigates the consequence of this covariance function
having the separable form given in \eqref{eq:kappasep}, where,
necessarily, both $\kappa_x$ and $\kappa_y$ are symmetric and
non-negative definite.

Consider the sequence of stochastic processes indexed by $n$,
\begin{equation}
  \label{eq:Fn}
  F^{(n)}(x, y) = \sum_{i=1}^n \sum_{j=1}^n Z_{ij} \, g_i(x) h_j(y)
\end{equation}
where the $\{ Z_{ij} \}$ are orthonormal, i.e.\ $\E(Z_{ij}) = 0$, $\E(
Z_{ij} Z_{i'j'}) = \delta_{ii'} \delta_{jj'}$ ($\delta$ is the
Kronecker delta), and where the functions in $\{ g_i \}$ and $\{ h_j
\}$ are continuous on $\curlyX$ and $\curlyY$, respectively.  While
there are no restrictions on $\{ g_i \}$ and $\{ h_j \}$ beyond
continuity, there is no loss of generality in removing obvious
redundancies.  Therefore we may assume that the functions are mutually
scaled so that $\Vert g_1 \Vert^2 \ldef \int g_1(x)^2 \, \rd x = 1$,
although in fact this property is not used below.  Also, we could
remove functions that are identically zero, but keeping them in allows
us to use just one limit for both $i$ and $j$, simplifying the
notation slightly.

The following two propositions together establish the equivalence
between \eqref{eq:Fn} and separability of the covariance function of
$F$.

\begin{proposition}\label{thm:QM}\samepage
  If $n$ is finite or $F^{(n)} \qm F$ uniformly on $\curlyX \times
  \curlyY$ then $F$ has a continuous separable covariance function.
\end{proposition}

\begin{proof}
  Only the $n \to \infty$ result needs to be proved; $n$ finite is a
  special case.  The convergence of $F^{(n)}(x, y)$ to $F(x, y)$ for each
  $(x, y)$ implies the pointwise convergence of the covariance
  functions; this is a standard continuity property of Hilbert spaces
  \citep[see, e.g.,][Lemma~3.2-2]{kreyszig78}.  Thus
  \begin{equation}
    \kappa \{ (x, y), (x', y') \} = \lim_{n\to\infty} \kappa^{(n)} \{ (x, y), (x', y') \}
  \end{equation}
  for each $(x, y)$ and $(x', y')$, where
  \begin{equation}
    \label{eq:kappan}
    \begin{split}
      \kappa^{(n)}  \{ (x, y), (x', y') \}
      & = \ip{ \boost F^{(n)}(x, y), F^{(n)}(x', y') } \\
      & = \sum_{i=1}^n \sum_{j=1}^n g_i(x) h_j(y) g_i(x') h_j(y') \\
      & = \sum_{i=1}^n g_i(x) g_i(x') \, \sum_{j=1}^n h_j(y) h_j(y') \\
      & = \kappa_x^{(n)}(x, x') \, \kappa_y^{(n)}(y, y') ,
    \end{split}
  \end{equation}
  say, where the second line follows from the orthonormality of the
  $\{ Z_{ij} \}$, and the functions $\kappa_x^{(n)}$ and
  $\kappa_y^{(n)}$ in the final line are clearly symmetric and
  non-negative definite (this proves the $n$ finite case).  The
  separability of $\kappa$ follows immediately.

  For continuity, $\kappa^{(n)}$ is uniformly convergent, because all
  random quantities have finite second moments and $F^{(n)}$ is uniformly
  convergent.  As $\kappa^{(n)}$ is continuous, uniform convergence
  implies that the limit $\kappa$ is continuous.
\end{proof}

Note that $\{ Z_{ij} \}$ must be uncorrelated, but the components do
not have to be standardised.  However, if the variance of $Z_{ij}$
depends on $(i,j)$, then it must factorise as $\lambda_i \gamma_j$ in
order for $F$ to have a separable covariance function; but in that
case the terms in the variance can be absorbed into $\{ g_i \}$ and
$\{ h_j \}$.

The second proposition asserts the converse.

\begin{proposition}\label{thm:KL}\samepage
  If $F$ has a continuous separable covariance function, then there
  exist sets of continuous functions $\{g_i\}$ and $\{h_j\}$ in
  \eqref{eq:Fn} such that $F^{(n)} \qm F$ uniformly on $\curlyX \times
  \curlyY$.
\end{proposition}

\begin{proof}
  This follows from an application of Mercer's Theorem and the
  Karhunen-Lo\`{e}ve expansion; these are both derived in
  \citet[Appendix]{ash65}.

  Mercer's Theorem states that if $\kappa_x$ is continuous on $\curlyX
  \times \curlyX$ then
  \begin{equation}
    \label{eq:Mercer}
    \kappa_x(x, x') = \lim_{n \to \infty} \sum_{i=1}^n \lambda_i\, \psi_i(x) \psi_i(x')
  \end{equation}
  where $\{ \lambda_i \}$ are the positive eigenvalues of $\kappa_x$
  and $\{ \psi_i(x) \}$ are the corresponding eigenfunctions, which
  are continuous on $\curlyX$; the convergence is absolute and uniform
  on $\curlyX \times \curlyX$.  Similarly, for $\kappa_y$,
  \begin{equation}
    \kappa_y(y, y') = \lim_{n \to \infty} \sum_{j=1}^n \gamma_j\, \phi_j(y) \phi_j(y') .
  \end{equation}
  If the covariance function $\kappa$ is separable, then, in obvious
  notation,
  \begin{equation}
    \label{eq:eigen}
    \begin{split}
    \kappa \{ (x, y), (x', y') \}
    & = \kappa_x(x, x') \, \kappa_y(y, y') \\
    & = \big( \lim_{n \to \infty} \kappa_x^{(n)}(x, x') \big) \big( \lim_{n \to \infty} \kappa_y^{(n)}(y, y') \big) \\      
    & = \lim_{n \to \infty} \big( \kappa_x^{(n)}(x, x') \, \kappa_y^{(n)}(y, y') \big) \\
    & = \lim_{n \to \infty} \sum_{i=1}^n \sum_{j=1}^n \lambda_i\, \gamma_j\, \psi_i(x) \phi_j(y) \psi_i(x') \phi_j(y') .
    \end{split}
  \end{equation}
  The series for $\kappa$ is absolutely convergent because both
  $\kappa_x$ and $\kappa_y$ are absolutely convergent (according to
  Mercer's Theorem).  The series is uniformly convergent because both
  $\kappa_x$ and $\kappa_y$ are uniformly convergent (according to
  Mercer's Theorem) and bounded.

  It is easy to verify that, for every $i$ and $j$, $\lambda_i
  \gamma_j$ is a positive eigenvalue for $\kappa$, and $\psi_i(x)
  \phi_j(y)$ a corresponding eigenfunction.  Thus we can apply the
  Karhunen-Lo\`{e}ve (KL) expansion.  Therefore, define
  \begin{equation}
    \label{eq:2}
    Z'_{ij} \ldef \iint_{\curlyX \times \curlyY} F(x, y) \psi_i(x) \phi_j(y) \, \rd x \, \rd y
  \end{equation}
  for which $\E( Z'_{ij} Z'_{i'j'}) = \lambda_i \gamma_j \delta_{ii'}
  \delta_{jj'}$.  It follows that if
  \begin{equation}
    \label{eq:Fn1}
    F^{(n)}(x, y) \ldef \sum_{i=1}^n \sum_{j=1}^n Z'_{ij} \, \psi_i(x) \phi_j(y) , 
  \end{equation}
  then $F^{(n)} \qm F$ uniformly on $\curlyX \times \curlyY$.
  Eq.~\eqref{eq:Fn1} has the required form, with $Z_{ij} \ldef Z'_{ij}
  / \sqrt{\lambda_i \gamma_j}$, $g_i(x) \ldef \sqrt{\lambda_i}
  \psi_i(x)$, and $h_j(y) \ldef \sqrt{\gamma_j} \phi_j(y)$.
\end{proof}

Putting these two Propositions together, we can conclude the
following.

\begin{proposition}[Representation theorem]\label{thm:summary}
  $F$ has continuous separable covariance function \emph{if and only
    if} it can be represented as $F^{(n)}$ in \eqref{eq:Fn}, or as its
  limit when $n \to \infty$.
\end{proposition}

This result is straightforward to derive in the special case where
both $\curlyX$ and $\curlyY$ are finite, and $F$ is Gaussian.  The
Hilbert space approach used here is necessary to lift these two
restrictions.  In the case where $F$ is a Gaussian process, the $\{
Z_{ij} \}$ are \emph{independent} standard Gaussian quantities, and
the convergence of $F^{(n)}$ to $F$ at each $(x, y)$ is almost sure; see,
\citet[p.~485]{loeve60}, or \citet[p.~279]{ash65}.


\medskip
\textbf{Generalisations.}  Two generalisations are immediate.  First,
the result is a special case of a more general result for complex $F$,
for which the inner product is $\ip{ F, F' } = \E( F \bar{F}' )$,
where $\bar{F}'$ is the complex conjugate of $F'$.  It is the complex
case that is treated in \citet[ch.~10]{loeve60}.  Second, the result
extends to any domain of $F$ with a finite number of dimensions, as
can be seen by inspecting the two proofs.  To apply directly the
results, the domain must be the product of closed and bounded
intervals of the real line.  However, more general versions of
Mercer's Theorem are available; see, e.g., \citet{ferreira09}.

\section{Products of processes}
\label{sec:products}

This section considers the special case in which $F$ can be written as
the product of two stochastic processes, one in $x$ and one in $y$:
\begin{equation}
  F(x, y) = F_x(x) \, F_y(y) .
\end{equation}
First, though, it is necessary to digress briefly on independence and
`uncorrelation', where this neologism (which is not original) is
shorter and also more direct than `lack of correlation'.

\subsection{Probabilistic independence and uncorrelation}
\label{sec:independence}

Consider two families of random quantities, $\{ X_i \}$ and $\{ Y_j
\}$.  Following \citet[ch.~4, sec.~3]{whittle00}, we say that these
two families are probabilistically independent if
\begin{equation}
  \label{eq:independent}
  \E \big[ g( \{ X_i \} ) \times h( \{ Y_j \}) \big]
  = \E\big[ g( \{ X_i\} ) \big] \times \E \big[  h( \{ Y_j\} ) \big]
\end{equation}
for all scalar functions $g$ and $h$ for which the righthand product
is defined.  This property is far too strong (i.e.\ restrictive) for
results that concern second-order properties such as covariances.
But, as shown below, simple uncorrelation is too weak.  Therefore
consider an indexed sequence of properties that runs from one to the
other.

\begin{definition}
  Two families of random quantities $\{ X_i \}$ and $\{ Y_j \}$ are
  \emph{$k$th-order uncorrelated} if
  \begin{displaymath}
    \E \Big( \prod_i X_i^{a_i} \times \prod_j Y_j^{b_j} \Big) = \E \Big( \prod_i X_i^{a_i} \Big) \times \E \Big( \prod_j Y_j^{b_j} \Big)
  \end{displaymath}
  for all tuples $\{ a_i \}$ and $\{ b_j \}$ comprising non-negative
  integers whose sum does not exceed $k$.
\end{definition}

If the families are first-order uncorrelated, then every $X_i$ is
uncorrelated with every $Y_j$, but nothing else is implied. At the
other end of the scale, $(k \to \infty)$th-order uncorrelated implies
probabilistic independence, if  $g$ and $h$ are restricted to functions
with well-behaved Taylor Series expansions.
%
%
Therefore,
statements of probabilistic independence are stupendously stronger
that those concerning second-order uncorrelation, noting that the set
of second-degree monomials is a vanishingly small fraction of the set
of all possible functions used in \eqref{eq:independent}.

\subsection{Products of processes}
\label{sec:products1}

Now we return to $F$s that are products of processes. It is a standard
and immediate result that if $F_x$ and $F_y$ are probabilistically
independent then $F$ has a separable covariance function; see, e.g.,
the textbooks of \citet[sec.~2.3]{stein99},
\citet[sec.~2.3]{santner03}, or \citet[sec.~4.2]{rasmussen06}. But in fact
probabilistic independence is far too strong: all that is required for
$F$ to have a separable covariance function is that the stochastic
processes $F_x$ and $F_y$ are second-order uncorrelated, so that
\begin{equation}
  \label{eq:2order}
  \E\{ F(x, y) \times F(x', y') \}
   = \E\{ F_x(x) F_x(x') \} \times \E\{ F_y(y) F_y(y') \} .
\end{equation}

One might imagine that the class of processes with separable
covariance functions contains many processes that cannot be
represented as products of second-order uncorrelated processes. In
general this is correct, but if we consider only the second-order
properties of the process then in fact the two classes are equivalent.

\begin{proposition}\label{thm:SOE}
  Every stochastic process with a separable covariance function is
  second-order identical to the product of second-order uncorrelated
  processes.
\end{proposition}

\begin{proof}
  It suffices to consider processes indexed by the tuple $(x, y)$, as
  the extension to more than two indices is immediate, so let $F(x,
  y)$ be a stochastic process with a separable covariance function.
  By Proposition~\ref{thm:summary}, $F(x, y)$ can be represented as
  \eqref{eq:Fn}, or its limit as $n \to \infty$.  Now replace each
  $Z_{ij}$ in \eqref{eq:Fn} with $Z_i Z_j'$.  In order to preserve the
  mean and covariance functions, these $\{ Z_i \}$ and $\{ Z_j' \}$
  must satisfy $\E( Z_i Z'_j ) = 0$ and $\E( Z_i Z_{i'} Z_j' Z_{j'}') =
  \delta_{ii'} \delta_{jj'}$.  The natural solution is that $\{ Z_i
  \}$ are orthonormal, $\{ Z_j'\}$ are orthonormal, and $\{ Z_i \}$
  and $\{ Z_j'\}$ are second-order uncorrelated. At this point, the
  original $F^{(n)}$ has been replaced by a new function with the same
  (separable) covariance function. But this new function factorises
  into the product
  \begin{displaymath}
    \Big( \sum_{i=1}^n Z_i \, g_i(x) \Big) \times \Big( \sum_{j=1}^n Z_j' \, h_j(y) \Big)
  \end{displaymath}
  and these two functions are second-order uncorrelated, because $\{
  Z_i \}$ and $\{ Z_j' \}$ are second-order uncorrelated.
\end{proof}

It is very important to appreciate that $\{ Z_{ij} \}$ and $\{ Z_i
Z'_j \}$ do not have the same joint distribution, and so replacing the
$n^2$ terms $\{ Z_{ij} \}$ with the $2n$ terms $\{ Z_i Z'_j \}$
changes the stochastic process to something other than $F$.  But this
new process has the same (zero) mean, and the same (separable)
covariance function, and so it is identical in its second-order
properties.  In general, the step where we replace $\{ Z_{ij} \}$ with
$\{ Z_i Z_j' \}$ shows that there are an infinite number of possible
candidates for $F_x(x) F_y(y)$.

To give an important example of the difference between $F(x, y)$ with
a separable covariance function and its second-order identical
$F_x(x) F_y(y)$, consider the case where $F$ is a Gaussian process.
In this case, as mentioned in Section~\ref{sec:result}, $\{ Z_{ij} \}$
are IID standard Gaussian random quantities. But if $Z_{ij} = Z_i
Z'_j$ then $Z_i$ and $Z'_j$ \emph{cannot} be Gaussian random
quantities, and in this case the implied $F_x$ and $F_y$ are
\emph{not} Gaussian processes, and nor is the product $F_x(x) F_y(y)$.
So a second-order identical process for a Gaussian process with a
separable covariance function is \emph{not} a Gaussian process.  It is
a different stochastic process that just happens to coincide with
$F(x, y)$ in its mean and covariance functions.


Proposition~\ref{thm:SOE} provides the explanation for the strong
constraints implied by a separable covariance function, presented in
Section~\ref{sec:separable}. Both \eqref{eq:ohagan} and
\eqref{eq:cressie} concern second-order properties, and, according to
Proposition~\ref{thm:SOE}, at this level $F$ will behave identically
to the product of second-order uncorrelated processes. When
considering $F(x, y)$ along $y$ at a given $x$, the product form shows
that the only effect of $x$ is to provide an uncertain scaling term
$F_x(x)$, which cancels in the correlation, hence \eqref{eq:cressie}.
The heuristic explanation of \eqref{eq:ohagan} is that under a product
structure for $F$ no information passes along diagonals in $(x, y)$.
This emphasises the point made in section~\ref{sec:separable}, that a
separable covariance function insists on a preferential set of
directions in the input space, aligned with the axes.

Figure~\ref{fig:summary} gives a summary of the results in this paper.

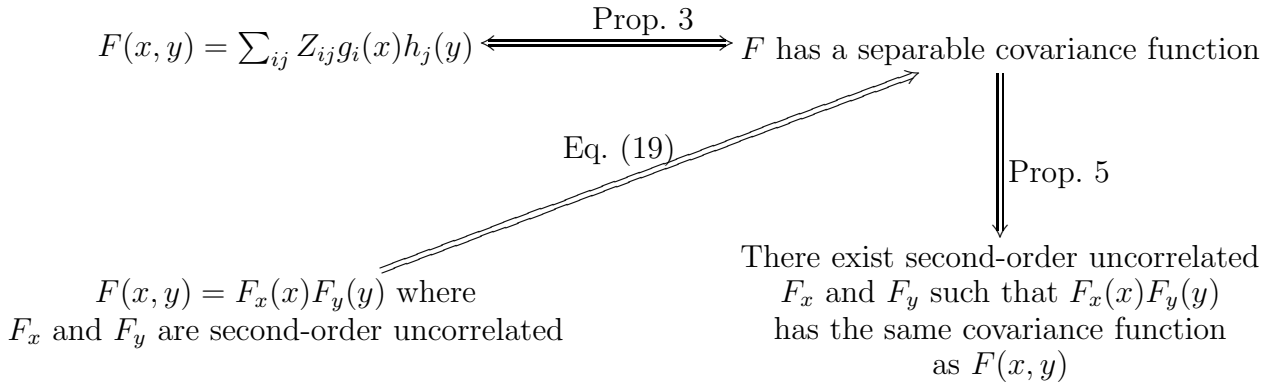
\begin{figure}
  \centering
  \begin{displaymath}
    \begin{array}{c}
      \xymatrix@+=60pt{
        \txt{$F(x, y) = \sum_{ij} Z_{ij} g_i(x) h_j(y)$}
        \ar@{<=>}[r]^{\txt{Prop.~\ref{thm:summary}}}
        & \txt{$F$ has a separable covariance function}
        \ar@{=>}[d]^{\txt{Prop.~\ref{thm:SOE}}}
        \\
        \txt{$F(x, y) = F_x(x) F_y(y)$ where\\$F_x$ and $F_y$ are second-order uncorrelated}
        \ar@{=>}[ur]^{\txt{Eq.~\eqref{eq:2order}}}
        & \txt{There exist second-order uncorrelated\\$F_x$ and $F_y$ such that $F_x(x) F_y(y)$\\has the same covariance function\\as $F(x, y)$} \\
      }
    \end{array}
  \end{displaymath}
  \caption{A summary of the results in this paper.}
  \label{fig:summary}
\end{figure}

\section{Implications for emulators}
\label{sec:emulators}

An emulator is a statistical representation of a function; denote this
function as $f$, assuming, for simplicity, that it is a deterministic
function of two arguments $x$ and $y$.  Typically, $f$ would be a
computer code and $f(x, y)$ would be expensive to run.  An emulator
offers the opportunity to augment the ensemble of runs with additional
judgements, for example about the monotonicity and smoothness of $f$.
In the Bayesian approach to emulation initiated by \citet{currin91},
one proposes a prior stochastic process for $f$, $F$ say, which
represents these additional judgements, and then conditions this
process on the ensemble of runs.

In the Bayesian approach, the prior stochastic process for $f$ is
written as the sum of two uncorrelated components, a set of regression
terms and a `residual':
\begin{equation}
  \label{eq:4}
  F(x, y) = \sum\nolimits_i \beta_i \, r_i(x, y) + E(x, y)
\end{equation}
where $\{ \beta_i \}$ are unknown regression coefficients, $\{ r_i \}$
are specified regressors, and $E(x, y)$ is mean-zero stochastic
process \citep[see, e.g.][ch.~2]{santner03}.  The separability of the
covariance function of $E$ was proposed in the early papers of of
\citet{sacks89} and \citet{currin91}, and is a crucial feature in
screening experiments designed to identify important inputs
\citep{welch92}.  It is now a standard choice, although
\citet{rougier09tiegcm} provide an example where prior information
about $f$ leads to a non-separable covariance function for $E$.  In
multivariate emulation, \citet{rougier08ope} proposed an $E$ which is
separable between inputs and outputs, but not necessarily separable
within the inputs.


\subsection{The role of the regressors}
\label{sec:regressors}

According to the representation theorem, including regressors with
\text{$\Var(\beta_i) > 0$} is sufficient to prevent the covariance
function of $F$ from being separable.  However, conventional wisdom
originating in the experiments of \citet[p.~16]{welch92} suggests that
regression terms beyond a mean effect are not required \citep[see
also][]{steinberg04}.  Furthermore, the mean effect is often estimated
(e.g.\ with its updated mean) and then plugged in.  This leaves us
with a prior emulator with mean zero and a separable covariance
function, or `nearly separable' if $E$ accounts for most of the prior
variance of $F$.

This is where Proposition~\ref{thm:SOE} comes in.  If our judgements
only extend to second-order---and it would be unusual to have
higher-order judgements about a complex computer code---then this is
akin to asserting that, as far as our judgements about the code are
concerned, there exist functions $f_x$ and $f_y$ for which $f(x, y)
\approx f_x(x) f_y(y)$.

Now consider the implications of this.  Were I to believe that there
existed $f_x$ and $f_y$ such that $f(x, y) \approx f_x(y) f_y(y)$ then
I would see little need for multi-parameter perturbations in the
ensemble of training runs.  Instead, for an efficient design I would
fix $x$ at $x_0$, and run the sequence $(x_0, y_1), (x_0, y_2),
\dots$; this would give me an accurate picture of the function $f_y$
up to the multiplicative constant $f_x(x_0)$.  Then I would reverse
the process, fixing $y$ at $y_0$.

But would anyone advocate this kind of experiment for a complex
computer code?  I doubt it: the standard experimental designs are
multi-parameter perturbations such as Latin Hypercube Designs
\citep[LHDs, see, e.g.,][ch.~5]{santner03}.  Now a LHD will perform no
worse than single parameter perturbations in the case where $f(x, y)
\approx f_x(x) f_y(y)$, and would be preferred for robustness.  But
few if any statisticians working in the field of computer experiments
would believe that a LHC would perform \emph{no better} than single
parameter perturbations.  And yet that is what is suggested by a prior
for $f$ with a separable covariance function.

This line of thought sheds some light on the $n = 10p$ rule ($n$ being
the number of runs, and $p$ being the number of inputs), which has
recently been reviewed, investigated, and advocated by \citet[``a
reasonable rule of thumb for an initial experiment'',
p.~374]{loeppky09}.  \textit{A priori}, this seems rather a small
number of runs, especially for more than six inputs (implying more
corners than runs, so that it is impossible for the convex hull of the
ensemble to fill the input space).  And so $n = 10p$ is an interesting
and potentially very useful rule.  However, its linearity in $p$ is
suggestive: this is exactly the kind of rule that would be appropriate
if $f(x, y) \approx f_x(x) f_y(y)$.  The value $10$ sounds about right
to fit a smooth curve for each of $f_x$ and $f_y$.

A close examination of the \citeauthor{loeppky09}\ experiment (their
section~5) reveals that all candidate functions on which this rule was
evaluated were sampled from a Gaussian process with a separable
covariance function.
So this experiment only ever considered the case of functions that
were second-order identical to $f_x(x) f_y(y)$.  We must conclude that
this experiment provides no support for $n = 10p$ in the general case.

\subsection{The effect of conditioning}
\label{sec:conditioning}

Let us put prior judgements aside, in favour of pragmatism.  Ideally,
the separability of the prior covariance function for $F$ would be a
property similar to prior stationarity: a convenient way to specify a
stochastic process with a small number of hyperparameters, with
possibly undesirable properties that are erased by conditioning on one
or more members of the ensemble.  This is in fact the case, as can
easily be seen from the representation theorem.  Conditioning the
prior for $F$ on a value for $f(x, y)$ induces a linear constraint
across the $\{ Z_{ij} \}$ in \eqref{eq:Fn}, and consequently the
components of $\{ Z_{ij} \}$ can no longer be uncorrelated.

Thus the use of a separable or nearly separable prior covariance
function for the emulator is defensible even though we judge that
$f(x, y)$ is much more complicated than $f_x(x) f_y(y)$, in the same
way that the use of a stationary covariance function is defensible
even though we are much more uncertain about $f$ around the edges of
the input space than in the middle (say).

Having said that, my personal view is that we should always include a
reasonable number of regression terms with uncertain coefficients in
the emulator, a point made in \citet{rougier09tiegcm}.  Conditioning
will erase second-order properties of the residual $E$ in and around
the convex hull of the ensemble of runs.  However, away from this
convex hull the updated $E$ will revert gradually to its prior
formulation.  If we can be confident that the ensemble is large enough
to fill the input space, then the prior choices we make for $E$
(stationarity, separability of the covariance function) will not
matter in practice.

But for really large applications, including many environmental
science applications like climate modelling, long run times and large
input spaces can imply that most of the input space is outside the
convex hull of the ensemble.  In this case, an emulator without
regression terms could revert to its prior around the edges of the
input space, but an emulator with regression terms is able to carry
the information in the ensemble all the way to the edges of the input
space.  An updated emulator without regressors would revert to a
separable covariance function. It is not clear to me what the effect
of this would be, e.g.\ in summaries that integrate over the input
space \citep{ooh04}.  But since the representation theorem shows that
separability of the covariance function is a strong constraint on the
structure of $F$, it seems wise not to impose it \textit{a priori}.

\section{Summary}
\label{sec:summary}

Probabilistic inference is extremely demanding, and we often find
ourselves making pragmatic choices where our judgements are only
partial.  This is certainly the case in a fully probabilistic
inference, but it is also true at second-order.  This paper has
examined choices about covariance functions, and, in particular, the
effect of the pragmatic choice to treat the covariance function as
separable---i.e.\ having a product structure.  It is well-known that
such a choice constrains conditional variances and marginal
correlations.  What was not known was the relationship between this
choice and the underlying stochastic process.  This paper has
completely resolved this issue, by providing a representation theorem
for stochastic processes with separable covariance functions.
Briefly, the centred process $F$ has a separable continuous covariance
function \textit{if and only if} it can be represented as $\sum_{ij}
Z_{ij} g_i(x) h_j(y)$ where $\{ Z_{ij} \}$ is a collection of mean
zero, variance one, uncorrelated quantities, and $\{ g_i \}$ and $\{
h_j \}$ are collections of continuous functions.

One use of this representation theorem is to provide a partial
converse to the standard result that if $F(x, y)$ can be represented
as $F_x(x) F_y(y)$, where $F_x$ and $F_y$ are probabilistically
independent, then $F$ has a separable covariance function.  By
substituting $\{ Z_i Z_j' \}$ for $\{ Z_{ij} \}$ in the representation
theorem it was shown that the second-order properties of an $F$ with
separable covariance function can be duplicated by the product of two
uncorrelated processes.  To get the most general statement of the
converse result it was necessary to introduce `$k$-fold uncorrelated'
families of random variables.  This converse result clarifies the
properties of stochastic processes with separable covariance
functions, by envisaging such processes as the product of uncorrelated
processes.  The theoretical results of this paper are summarised in
Figure~\ref{fig:summary}.

The main relevance of these results is in the emulation of complex
computer simulators, part of the statistical field of computer
experiments.  In this application it is completely standard to
represent a large chunk of the prior variance of the emulator in the
form of a stochastic process with a separable covariance function.
Indeed, the conventional wisdom is that the whole prior may be thus
represented.  The results of this paper suggests that this
conventional choice is in fact highly restrictive, being equivalent to
the judgement that the simulator $f(x, y)$ could be approximated by
the product $f_x(x) f_y(y)$.  This has practical implications for
experimental design, and casts doubt upon the provenance of the $n =
10 p$ rule for selecting sample size.  The representation theorem also
shows that it is very easy to construct emulators which do not have
separable covariance functions (by including regression terms with
uncertain coefficients), for which there is no \textit{a priori}
restriction to $f(x, y) \approx f_x(x) f_y(y)$.

These new theoretical results notwithstanding, in many applications
the use of a prior emulator with a separable covariance function is
innocuous and will continue.  This is because updating the emulator
with one or more runs of the computer code will erase the separability
of the covariance function, in the same way that other properties such
as stationary are also erased.  This is demonstrated through the
representation theorem.  The main concern is then for large
experiments, where the ensemble of simulator runs does not fill the
input space, and for which emulators based around a separable
covariance function may revert to their prior at the edges and corners
of the input space (remembering that a high-dimensional space is all
edges and corners).

\section*{Acknowledgements}

The outline of this result was sketched while I was participating in
the Isaac Newton Institute programme \textit{Mathematical and
  Statistical Approaches to Climate Modelling and Prediction},
Aug--Dec 2010. I would like to thank Natalia Bochkina for her help in
part of the proof of Proposition~\ref{thm:KL}, and Marc Genton and
Tony O'Hagan for their comments on a previous and shorter version of
this paper.  I would also like to thank participants of the Isaac
Newton Institute workshop \textit{Accelerating Industrial Productivity
  via Deterministic Computer Experiments and Stochastic Simulation
  Experiments}, part of the programme \textit{Design and Analysis of
  Experiments}, Jul--Dec 2011, for very stimulating discussions.


\end{document}